\documentclass[12pt,leqno]{amsart}
\usepackage{amsmath}
\usepackage{amsthm}
\usepackage{amssymb}
\def\overset#1#2{{\mathrel{\mathop {{#2}_{}}\limits^{#1}}}}
\def\underset#1#2{{\mathrel{\mathop {{}_{} {#2}}\limits_{{#1}_{}}}}}
\def\upplim_#1{\underset{#1}{\overline\lim}\;}
\def\lowlim_#1{\underset{#1}{\underline\lim}\;}
\newtheorem{claim}[equation]{\indent \rm {\it Claim}}
\newtheorem{definition}[equation]{Definition}

\newtheorem{lemma}[equation]{Lemma}

\newtheorem{proposition}[equation]{Proposition}

\newtheorem{theorem}[equation]{Theorem}

\newcommand{\C}{{\mathbb{C}}}

\newcommand{\N}{\mathbb{N}}
\newcommand{\B}{\mathbb{B}}

\renewcommand{\P}{{\mathbb{P}}}

\newcommand{\R}{{\mathbb{R}}}

\newcommand{\Z}{\mathbb{Z}}
\numberwithin{equation}{section}

\begin{document}
\title[Meromorphic mappings of a complete K\"{a}hler manifolds]{Degeneracy theorems for meromorphic mappings of complete K\"{a}hler manifolds sharing hyperplanes in projective spaces} 

\author{Si Duc Quang$^{1,2}$}
\address{$^{1}$Department of Mathematics, Hanoi National University of Education\\
136-Xuan Thuy, Cau Giay, Hanoi, Vienam.}
\address{$^{2}$Thang Long Institute of Mathematics and Applied Sciences\\
Nghiem Xuan Yem, Hoang Mai, Ha Noi.}
\email{quangsd@hnue.edu.vn}

\def\thefootnote{\empty}

\begin{abstract} {Let $M$ be a complete K\"{a}hler manifold, whose universal covering is biholomorphic to a ball $\B^m(R_0)$ in $\C^m$ ($0<R_0\le +\infty$). In this article, we will show that if three meromorphic mappings $f^1,f^2,f^3$ of $M$ into $\P^n(\C)\ (n\ge 2)$ satisfy the condition $(C_\rho)$  and share $q\ (q> C+\rho K)$ hyperplanes in general position regardless of multiplicity with certain positive constants $K$ and $C <2n$ (explicitly estimated), then there are some algebraic relation between them. A degeneracy theorem for the product of $k\ (2\le k\le n+1)$ meromorphic mappings sharing hyperplanes is also given. Our results generalize the previous results in the case of  meromorphic mappings from $\C^m$ into $\P^n(\C)$.}
\end{abstract}
\maketitle

\section{Introduction}

Let $f$ be a linearly nondegenerate meromorphic mapping of $\C^m$ into $\P^n(\C)$, $d$ be a positive integer and $H_1,\ldots,H_q$ be $q$ hyperplanes of $\P^n(\C)$ in general position with
$$\dim  f^{-1}(H_i \cap H_j) \le m-2 \quad (1 \le i<j \le q).$$
We consider the set $\mathcal {F}(f,\{H_i\}_{i=1}^q,d)$ of all linearly nondegenerate meromorphic mappings $g: \C^m \to \P^n(\C)$ satisfying the following two conditions:

\ (a)\ $\min\{\nu_ {(f,H_j)}(z),d\}=\min\{\nu_{(g,H_j)}(z),d\}\quad (1\le j \le q),$

\ (b) \ $f(z) = g(z)$ on $\bigcup_{j=1}^{q}f^{-1}(H_j)$.\\
Here $\nu_{(f,H)}(z)$ stands for the intersecting multiplicity of the image of $f$ with a hyperplane $H$ at the point $f(z)$. Hence, $\nu_{(f,H)}$ may be considered as  the divisor $f^*H$. If $d=1$, we will say that $f$ and $g$ share $q$ hyperplanes $\{H_j\}_{j=1}^q$ regardless of multiplicity.

In 1988, S. Ji \cite{J} showed that if $n\ge 2$, then the map $f^1\times f^2\times f^3 :\C^m\longrightarrow \P^n(\C)\times\P^n(\C)\times\P^n(\C)$ is algebraically degenerate for every three maps $f^1,f^2,f^3\in\mathcal F(f,\{H_i\}_{i=1}^{3n+1},1)$. Later,  in 1998, H. Fujimoto \cite{F98} proved a degeneracy theorem for $n+2$ meromorphic mappings sharing $2n+2$ hyperplanes with multiplicities are counted to level $\dfrac{n(n+1)}{2}+n$ as follows.

\vskip0.2cm
\noindent
\textbf{Theorem A.}\  {\it Suppose that $q\ge 2n+2$ and $d=\dfrac{n(n+1)}{2}+n$ and take arbitrary $n+2$ mappings $f_1,\dots , f_{n+2}$ in $\mathcal F(f,\{H_i\}_{i=1}^q,d).$ Then, there are $n +1$ hyperplanes $H_{i_0} , . . . , H_{i_n}$ among $H_i'$s such that for each pair $(j,k)$ with $0\le i<k\le n$, we have that
$$
\frac{(f^2,H_{i_j})}{(f^2,H_{i_k})}-\frac{(f^1,H_{i_j})}{(f^1,H_{i_k})},\frac{(f^3,H_{i_j})}{(f^3,H_{i_k})}-\frac{(f^1,H_{i_j})}{(f^1,H_{i_k})},\cdots ,\frac{(f^{n+2},H_{i_j})}{(f^{n+2},H_{i_k})}-\frac{(f^1,H_{i_j})}{(f^1,H_{i_k})}$$
are linearly dependent.}

\vskip0.2cm
The above results of L. Smiley and H. Fujimoto have been extended by many authors, such as  \cite{QQ,CY} and others.

Recently, S. D. Quang in \cite{Q1} proved a stronger result as follows.

\vskip0.2cm
\noindent
{\bf Theorem B.}\ {\it If $n\ge 2$, then the family $\mathcal F(f, \{H_i\}_{i=1}^{2n+2}, 1)$ contains at most two maps.}

\vskip0.2cm
This result have covered all previous results on the degeneracy and algebraic dependence problem of meromorphic mappings sharing at least $2n+2$ hyperplanes in general position of $\P^n(\C)$. In \cite{NQ}, N. T. Nhung - L. N. Quynh firstly showed a algebraic relation between mappings which share less than $2n+2$ hyperplanes in general position regardless of multiplicities as follows.

\vskip0.2cm
\noindent
{\bf Theorem C.}\ {\it Let $f$ be a linearly non-degenerate meromorphic mapping of $\C^m$ into $\P^n(\C)$ and let $H_1,\ldots,H_q$ be $q$ hyperplanes of $\P^n(\C)$ in general position such that
$$ \dim f^{-1}(H_i)\cap f^{-1}(H_j)\le m-2, \ \forall 1\le i<j\le q. $$ 
Let $f^1,f^2,f^3$ be three maps in $\mathcal F(f,\{H_i\}_{i=1}^q,1)$. Assume that $q\ge\frac{n+6+\sqrt{7n^2+2n+4}}{2}$. Then there exist $[\frac{q}{2}]$ hyperplanes $H_{i_1},\ldots,H_{i_{[\frac{q}{2}]}}$ among $H_i'$s such that:
$$\frac{(f^1,H_{i_j})}{(f^1,H_{i_1})}=\frac{(f^2,H_{i_j})}{(f^2,H_{i_1})}\text{ or } \frac{(f^2,H_{i_j})}{(f^2,H_{i_1})}=\frac{(f^3,H_{i_j})}{(f^3,H_{i_1})}\text{ or }
\frac{(f^3,H_{i_j})}{(f^3,H_{i_1})}=\frac{(f^1,H_{i_j})}{(f^1,H_{i_1})},$$
for every $j\in\{2,\ldots ,[\frac{q}{2}]\}$.}

Recently, by introducing the notion of ``\textit{functions of small integration}'' with respect to meromorphic mappings on K\"{a}hler manifold, in \cite{Q3} we have generalized Theorem B to the case of meromorphic mappings of a complete K\"{a}hler manifold into $\P^n(\C)$. Motivated our method in \cite{Q3}, in this paper, we will generalize Theorem C to the case of meromorphic mappings of a complete K\"{a}hler manifold.

To state our first main result, we need to recall the following.

Let $M$ be an $m$-dimensional complete K\"{a}hler manifold with K\"{a}hler form $\omega$ and $f$ be a meromorphic map of $M$ into $\P^n(\C)$. Throughout this paper, we always assume that the universal covering of $M$ is biholomorphic to a ball $\B^m(R_0)$ in $\C^m$\ ($0<R_0\le +\infty$). For $\rho\ge 0$, we
say that $f$ satisfies the condition ($C_\rho$) if there exists a nonzero bounded continuous real-valued function $h$ on $M$ such that
$$\rho\Omega_f+dd^c\log h^2\ge \mathrm{Ric}\omega,$$
where $\Omega_f$ denotes the pull-back of the Fubini-Study metric form on $\P^n(\C)$ by $f$.

Let $f$ be a linearly non-degenerate meromorphic mapping from $M$ into $\P^n(\C)$ which satisfies the condition $(C_\rho)$.
Let $H_1,\ldots,H_q$ be $q$ hyperplanes of $\P^n(\C)$ in general possition. 
Assume that 
$$\dim f^{-1}(H_i)\cap f^{-1}(H_j) \le m-2 \quad (1 \le i<j \le q).$$ 
The family $\mathcal {F}(f,\{H_i\}_{i=1}^q,d)$ is defined similarly as above.

Our main result is stated as follows.

\begin{theorem}\label{thm1.1}
Let $M$ be an $m$-dimensional complete K\"{a}hler manifold whose universal covering is biholomorphic to a ball $\B^m(R_0)$ in $\C^m$\ ($0<R_0\le +\infty$), and let $f$ be a linearly non-degenerate meromorphic mapping of $M$ into $\P^n(\C)\ (n\ge 2)$. Let $H_1,\ldots,H_q$ be $q$ hyperplanes of $\P^n(\C)$ in general possition. Assume that $f$ satisfies the condition $(C_\rho)$ and
$$\dim f^{-1}(H_i)\cap f^{-1}(H_j) \le m-2 \quad (1 \le i<j \le q).$$ 
Let $f^1,f^2,f^3$ be three maps in $\mathcal F(f,\{H_i\}_{i=1}^q,1)$. Assume that 
$$q\ge\frac{n+6+(7n^2+2n+4)^{1/2}}{2}+\left (\rho\frac{3n((n+1)(q+n-3)+q-2)}{2}\right)^{1/2}.$$ 
Then there exist $[\frac{q}{2}]$ hyperplanes $H_{i_1},\ldots,H_{i_{[\frac{q}{2}]}}$ among $H_i'$s such that:
$$\frac{(f^1,H_{i_j})}{(f^1,H_{i_1})}=\frac{(f^2,H_{i_j})}{(f^2,H_{i_1})}\text{ or } \frac{(f^2,H_{i_j})}{(f^2,H_{i_1})}=\frac{(f^3,H_{i_j})}{(f^3,H_{i_1})}\text{ or }
\frac{(f^3,H_{i_j})}{(f^3,H_{i_1})}=\frac{(f^1,H_{i_j})}{(f^1,H_{i_1})},$$
for every $j\in\{2,\ldots ,[\frac{q}{2}]\}$.
\end{theorem}

Note: If $M=\C^m$ then we may choose $\rho=0$, and hence Theorem \ref{thm1.1} immediately implies Theorem C.

In the last section of this paper, we will prove a degeneracy theorem for a family of meromorphic mappings of a complete K\"{a}hler manifold sharing hyperplanes as follows.

\begin{theorem}\label{thm1.2}
Let $M$ be an $m$-dimensional complete K\"{a}hler manifold whose universal covering is biholomorphic to a ball $\B^m(R_0)$ in $\C^m$\ ($0<R_0\le +\infty$), and let $f$ be a linearly non-degenerate meromorphic mapping of $M$ into $\P^n(\C)\ (n\ge 2)$ satisfying the condition $(C_\rho)$. Let $H_1,\ldots,H_q$ be $q$ hyperplanes of $\P^n(\C)$ in general possition such that
$$\dim f^{-1}(H_i)\cap f^{-1}(H_j) \le m-2 \quad (1 \le i<j \le q).$$ 
Let $f^1,\ldots, f^k$ be $k$ mappings in $\mathcal {F}(f,\{H_i\}_{i=1}^q,n)$. Assume that 
$$q> n+1+\frac{knq}{kn+(k-1)q-k}+\rho\frac{kn(n+1)}{2}.$$
Then $f^1\times \cdots\times f^k$ is algebraic degenerate.
\end{theorem}

\section{Basic notions and auxiliary results from Nevanlinna theory}

\begin{definition}[Cartan's auxialiary function]\label{def2.2}
For meromorphic functions $F,G,H$ on $\B^m(R_0)$ and $\alpha =(\alpha_1,\ldots ,\alpha_m)\in \Z_+^m$, we define the Cartan's auxiliary function as follows:
$$
\Phi^\alpha(F,G,H):=F\cdot G\cdot H\cdot\left | 
\begin {array}{cccc}
1&1&1\\
\frac {1}{F}&\frac {1}{G} &\frac {1}{H}\\
\mathcal {D}^{\alpha}(\frac {1}{F}) &\mathcal {D}^{\alpha}(\frac {1}{G}) &\mathcal {D}^{\alpha}(\frac {1}{H})
\end {array}
\right|.
$$
\end{definition}

\begin{lemma}[{see \cite[Proposition 3.4]{F98}}]\label{lem2.3}
If $\Phi^\alpha(F,G,H)=0$ and $\Phi^\alpha(\frac {1}{F},\frac {1}{G},\frac {1}{H})=0$ for all $\alpha$ with $|\alpha|\le 1$, then one of the following assertions holds:

(i) \ $F=G, G=H$ or $H=F$,

(ii) \ $\frac {F}{G},\frac {G}{H}$ and $\frac {H}{F}$ are all constant.
\end{lemma}

\begin{lemma}[{see \cite[Lemma 3.2]{NQ}}]\label{lem2.4}
If $\Phi^\alpha(F,G,H)\equiv 0$ for all $|\alpha|=1$ then there exist constants $\alpha_0,\beta_0$, not all zeros, such that
$$ \alpha_0\left (\dfrac{1}{F}-\dfrac{1}{G}\right )=\beta_0 \left (\dfrac{1}{F}-\dfrac{1}{H}\right ).$$
\end{lemma}

In the following, we recall some notion and results on function of small integration and function of bounded integration with respect to meromorphic mappings from a ball $\B^m(\C)\subset\C^m$ into $\P^n(\C)$ due to \cite{Q2,Q3}.

Let $f$ be a meromorphic mappings from $\B^m(R)$ into $\P^n(\C)$ and let $\nu$ be a divisor on $\B^m(R)$. As usual, we denote by $T_f(r,r_0)$ and $N^{[M]}(r,r_0;\nu)\ (0<r_0<r<+\infty)$ the characteristic function of $f$ and the counting function of $\nu$ with trucated level $M$ respectively. For a meromorphic function $\varphi$ on $\B^m(\R)$, we denote by $\nu_\varphi^0$, $\nu_\varphi^\infty$ and $\nu_\varphi$ the divisor of zeros of $\varphi$, the divisor of poles of $\varphi$ and the divisor of $\varphi$ respectively. We set $N^{[N]}_{\varphi}(r,r_0)=N^{[N]}(r,r_0;\nu_{\varphi}^0)$.

Let $f^1,\ldots, f^k$ be $k$ meromorphic mappings from $\B^m(\C)$ into $\P^n(\C)$. For each $1\le u\le k$, fix a reduced representation $f^u=(f^u_0:\cdots :f^u_n)$ of $f^u$ and set $\|f^u\|=(|f^u_0|^2+\cdots+|f^u_n|^2)^{1/2}$.

Denote by $\mathcal C(\B^m(R_0))$ the set of all non-negative functions $g: \B^m(R_0)\to [0,+\infty]$ which are continuous on $\B^m(R_0)$ (corresponding to the topology of the compactification $[0,+\infty]$) outside an analytic set of codimension two and only attain $+\infty$ in an analytic thin set.

\begin{definition}[{Functions of small integration \cite[Definition 3.1]{Q3}}]\label{3.1}
A function $g$ in $\mathcal C(\B^m(R_0))$ is said to be of small integration with respective to $f^1,\ldots,f^k$ at level $l_0$ if there exist an element $\alpha=(\alpha_1,\ldots,\alpha_m)\in\N^m$ with $|\alpha|\le l_0$,  a positive number $K$, such that for every $0\le tl_0<p<1$,
$$\int_{S(r)}|z^\alpha g|^t\sigma_m \le K\left(\frac{R^{2m-1}}{R-r}\sum_{u=1}^kT_{f^u}(r,r_0)\right)^p$$
for all $r$ with $0<r_0<r<R<R_0$, where $z^\alpha=z_1^{\alpha_1}\cdots z_m^{\alpha_m}$. 
\end{definition}

Remark: We initially introduced this notion in \cite{Q2,Q3}. In \cite{Q2} (and the first version of \cite{Q3}), we only give the definition of ``function of small integration'' for non-negative plurisubharmonic functions, and hence the set of functions of small integration may not large enough to contain all neccesary functions. Therefore, in the final version of \cite{Q3} and this paper, we re-define this notion for all functions in $\mathcal C(\B^m(R_0))$. Of course, the results and the proofs in \cite{Q2} are not effected and still hold.

Denote by $S(l_0;f^1,\ldots,f^k)$ the set of all functions in $\mathcal C(\B^m(R_0))$ which are of small integration with respective to $f^1,\ldots,f^k$ at level $l_0$. We see that, if $g$ belongs to $S(l_0;f^1,\ldots,f^k)$ then $g$ is also belongs to $S(l;f^1,\ldots,f^k)$ for every $l>l_0$. Moreover, if $g$ is a constant function then $g\in S(0;f^1,\ldots,f^k)$.

\begin{proposition}[{see \cite[Proposition 3.2]{Q3}}]\label{prop2.6}
If $g_i\in S(l_i;f^1,\ldots,f^l)\ (1\le i\le s)$ then $\prod_{i=1}^sg_i\in S(\sum_{i=1}^sl_i;f^1,\ldots,f^l)$.
\end{proposition}

\begin{definition}[{Functions of bounded integration \cite[Definition 3.3]{Q3}}]\label{def2.7} 
A meromorphic function $h$ on $\B^m(R_0)$ is said to be of bounded integration with bi-degree $(p,l_0)$ for the family $\{f^1,\ldots,f^k\}$ if there exists $g\in S(l_0;f^1,\ldots,f^k)$ satisfying
$$|h|\le \|f^1\|^p\cdots \|f^u\|^p\cdot g,$$
outside a proper analytic subset of $\B^m(R_0)$.
\end{definition}
Denote by $B(p,l_0;f^1,\ldots,f^k)$ the set of all meromorphic functions on $\B^m(R_0)$ which are of bounded integration of bi-degree $(p,l_0)$ for $\{f^1,\ldots,f^k\}$. We have some properties:
\begin{itemize}
\item For a meromorphic mapping $h$, $|h|\in S(l_0;f^1,\ldots,f^k)$ iff $h\in B(0,l_0;f^1,\ldots,f^k)$.
\item $B(p,l_0;f^1,\ldots,f^k)\subset B(p,l;f^1,\ldots,f^k)$ for every $0\le l_0<l$.
\item If $h_i\in B(p_i,l_i;f^1,\ldots,f^k)\ (1\le i\le s)$ then 
$$h_1\cdots h_m\in B(\sum_{i=1}^sp_i,\sum_{i=1}^sl_i;f^1,\ldots,f^k).$$
\end{itemize}

\begin{proposition}[{see \cite[Proposition 3.5]{Q3}}]\label{prop2.9} 
	Let $M$ be a complete connected K\"{a}hler manifold whose universal covering is biholomorphic to a ball $\B^m(R_0)\ (0<R_0\le +\infty)$. Let $f^1,f^2,\ldots,f^k$ be $m$ linearly non-degenerate meromorphic mappings from $M$ into $\P^n(\C)$, which satisfy the condition $(C_\rho)$. Let $H_1,\ldots,H_q$ be $q$ hyperplanes of $\P^n(\C)$ in general position, where $q$ is a positive integer. Assume that there exists a non zero holomorphic function $h\in B(p,l_0;f^1,\ldots,f^k)$ such that
	$$\nu_h\ge\lambda\sum_{u=1}^k\sum_{i=1}^q\nu^{[n]}_{(f^u,H_i)},$$
	where $p,l_0$ are non-negative integers, $\lambda$ is a positive number. Then we have
	$$q\le n+1+\rho k\frac{n(n+1)}{2}+\frac{1}{\lambda}\left (p+\rho l_0\right).$$
\end{proposition}

\section{Proof of  Main Theorems}
\begin{lemma}\label{lem3.1}
Let $f$ be a meromorphic mapping from $\B^{m}(R_0)\ (0<R_0\le +\infty)$ into $\P^n(\C)$. Let $f^1,f^2,\ldots,f^k$ be three maps in $\mathcal F(f,\{H_i\}_{i=1}^q,1)$. Assume that each $f^u$ has a representation $f^u=(f^u_{0}:\cdots :f^u_{n})$, $1\le u\le k$. Suppose that there exist $1\le i_1<i_2<\cdots <i_k\le q$ such that
$$ 
P:=\mathrm{det}\left (\begin{array}{cccc}
(f^1,H_{i_1})&(f^1,H_{i_2})&\cdots&(f^1,H_{i_k})\\ 
\vdots&\vdots&\cdots&\vdots\\
(f^k,H_{i_1})&(f^k,H_{i_2})&\cdots&(f^k,H_{i_k})
\end{array}\right )\not\equiv 0.
$$
Then we have
\begin{align*}
\nu_P (z)\ge\sum_{j=1}^k(\min_{1\le u\le k}\{\nu_{(f^u,H_i)}(z)\}-\nu^{[1]}_{(f,H_i)}(z))+ (k-1)\sum_{i=1}^q\nu^{[1]}_{(f,H_i)}(z),
\end{align*}
for every $z\in\B^m(R_0)$ outside an analytic set of codimension two.
\end{lemma}
\begin{proof} Without loss of generalization, we suppose that $i_1=1,\ldots,i_k=k$. Consider a point $z\not\in\bigcup_{i\ne j}\left (f^{-1}(H_i)\cap f^{-1}(H_j)\right)$. If $z$ is a zero of some $(f,H_j)\ (1\le j\le k)$, for instance $z$ is a zero of $(f,H_1)$, then $z$ is zero of $(f^u,H_1)$ with multiplicity at least $\min_{1\le u\le k}\{\nu_{(f^u,H_1)}(z)\}$ and $z$ also is a zero of all $\dfrac{(f^u,H_j)}{(f^u,H_q)}-\dfrac{(f^1,H_j)}{(f^1,H_q)}$. We have $P=\left (\prod_{u=1}^k(f^u,H_q)\right )\det A,$ where
$$ 
A=\left (\begin{array}{cccc}
\dfrac{(f^1,H_1)}{(f^1,H_q)}&\dfrac{(f^2,H_1)}{(f^2,H_q)}-\dfrac{(f^1,H_1)}{(f^1,H_q)}&\cdots&\dfrac{(f^k,H_1)}{(f^k,H_q)}-\dfrac{(f^1,H_1)}{(f^1,H_q)}\\ 
\vdots&\vdots&\cdots&\vdots\\
\dfrac{(f^1,H_k)}{(f^1,H_q)}&\dfrac{(f^2,H_k)}{(f^2,H_q)}-\dfrac{(f^1,H_k)}{(f^1,H_q)}&\cdots&\dfrac{(f^k,H_k)}{(f^k,H_q)}-\dfrac{(f^1,H_k)}{(f^1,H_q)}
\end{array}\right ).
$$
Hence $z$ is a zero of all elements in the columns $2,3,\ldots,k$, and is also a zero of all elements in the first row of the matrix $A$ with multiplicity at least $\min_{1\le u\le k}\{\nu_{(f^u,H_1)}(z)\}$. This implies that
\begin{align*}
\nu_P(z)&\ge (k-1)+\left (\min_{1\le u\le k}\{\nu_{(f^u,H_1)}(z)\}-1\right)\\
&=\sum_{j=1}^k\left (\min_{1\le u\le k}\{\nu_{(f^u,H_i)}(z)\}-\nu^{[1]}_{(f,H_i)}(z)\right )+ (k-1)\sum_{i=1}^q\nu^{[1]}_{(f,H_i)}(z).
\end{align*}
Now, if $z$ is a zero of some $(f,H_j)$ with $j>k$. Without loss of generality we may suppose that $k<q$. We see that $z$ is zero of all elements in the columns $2,\ldots,k$. This implies that
$$\nu_P(z)\ge (k-1)=\sum_{j=1}^k\left (\min_{1\le u\le k}\{\nu_{(f^u,H_i)}(z)\}-\nu^{[1]}_{(f,H_i)}(z)\right )+ (k-1)\sum_{i=1}^q\nu^{[1]}_{(f,H_i)}(z).$$
The lemma is proved.
\end{proof}

 If $M=\C^m$ then we may choose $\rho=0$ and Theorem \ref{thm1.1} is exactly Theorem C in \cite{NQ}. Hence, in this proof of Theorem \ref{thm1.1} we only consider the case where $M=\mathbb B^m(1)$. 

Now for three mappings $f^1, f^2, f^3 \in \mathcal {F}(f,\{H_i\}_{i=1}^{q},1)$, we define:

$$F_k^{ij}=\dfrac {(f^k,H_i)}{(f^k,H_j)}\ (0\le k\le 2,\ 1\le i,j\le q).$$

\begin{lemma}[{see \cite[Lemma 4.7]{Q3}}]\label{lem3.2}
With the assumption of Theorem \ref{thm1.1}, let $f^1,f^2,f^3$ be three meromorphic mappings in $\mathcal F(f,\{H_i\}_{i=1}^{q},1)$.
Assume that there exist $i, j \in \{1,2,\ldots ,q\}$ and  $\alpha\in\N^m$ with $|\alpha|=1$ such that
$\Phi^{\alpha}_{ij}\not\equiv 0.$
Then there exists a meromorphic function $g_{ij}\in B(1,1;f^1,f^2,f^3)$ such that
\begin{align*}
\nu_{g_{ij}}\ge \sum_{u=1}^{3}\nu^{[n]}_{(f^u,H_i)}+\sum_{u=1}^3\nu^{[n]}_{(f^u,H_j)}+2\sum_{\overset{t=1}{t\ne i,j}}^{q}\nu^{[1]}_{(f,H_t)}-(2n+1)\nu^{[1]}_{(f,H_i)}-(n+1)\nu^{[1]}_{(f,H_j)}.
\end{align*}
Furthermore, if there exits $\alpha'\in\Z^m_{+}$ with $|\alpha'|=1$ such that $\Phi^{\alpha'}(F_1^{ji},F_2^{ji},F_3^{ji})\not\equiv 0$ then there exist a meromorphic function $g_{\{i,j\}}\in B(2,2;f^1,f^2,f^3)$ such that
$$ \nu_{g_{\{i,j\}}}\ge 2\sum_{u=1}^3\sum_{t=i,j}\nu^{[n]}_{(f^u,H_t)}+4\sum_{\overset{t=1}{t\ne i,j}}^{q}\nu^{[1]}_{(f,H_t)}-(3n+2)\sum_{t=i,j}\nu^{[1]}_{(f,H_t)}. $$
\end{lemma}

\begin{proof} The existence of the function $g_{ij}$ is proved in \cite[Lemma 4.7]{Q3}. We just show the existence of the function $g_{\{i,j\}}$.

Indeed, if there exits $\alpha'\in\Z^m_{+}$ with $|\alpha'|=1$ such that $\Phi^{\alpha'}(F_1^{ji},F_2^{ji},F_3^{ji})\not\equiv 0$ then we construct function $g_{ji}$ similarly as $g_{ij}$ and set $g_{\{i,j\}}=g_{ij}g_{ji}$. It is clear that $g_{\{i,j\}}\in B(2,2;f^1,f^2,f^3)$ and
$$ \nu_{g_{\{i,j\}}}\ge 2\sum_{u=1}^3\sum_{t=i,j}\nu^{[n]}_{(f^u,H_t)}+4\sum_{\overset{t=1}{t\ne i,j}}^{q}\nu^{[1]}_{(f,H_t)}-(3n+2)\sum_{t=i,j}\nu^{[1]}_{(f,H_t)}. $$

The lemma is proved.
\end{proof}

\begin{lemma}\label{lem3.8}
	Let $f$ and $H_1,\ldots,H_q$ be as in Theorem \ref{thm1.1}. Let $f^1,f^2,f^3$ be three maps in $\mathcal F(f,\{H_i\}_{i=1}^q,1)$. Assume that 
$$q> n+1+\frac{3nq}{2q+3n-6}+\rho\left (\frac{n(n+1)}{2}+\frac{3nq}{2q+3n-6}\right).$$ 
Then there exist $([\frac{q}{2}]+1)$ hyperplanes $H_{i_0},\ldots,H_{i_{[\frac{q}{2}]}}$ among $H_i'$s such that for each $j\ (1\le j\le [\frac{q}{2}])$ there exist two constants $\alpha_j,\beta_j,$ not all zeros, satisfying
	\begin{align*}
		&\alpha_j\left (\frac{(f^1,H_{i_j})}{(f^1,H_{i_0})}-\frac{(f^2,H_{i_j})}{(f^2,H_{i_0})}\right )=\beta_j \left (\frac{(f^1,H_{i_j})}{(f^1,H_{i_0})}-\frac{(f^3,H_{i_j})}{(f^3,H_{i_0})}\right )\\
		\mathrm{or}\ \ &\alpha_j\left (\frac{(f^1,H_{i_0})}{(f^1,H_{i_j})}-\frac{(f^2,H_{i_0})}{(f^2,H_{i_j})}\right )=\beta_j \left (\frac{(f^1,H_{i_0})}{(f^1,H_{i_j})}-\frac{(f^3,H_{i_0})}{(f^3,H_{i_j})}\right ).
	\end{align*}
\end{lemma}
\begin{proof}
For each $i\ (1\le i\le q)$, we denote by $S(i)$ the set of all $j\ne i$ such that $\Phi^{\alpha}(F_1^{ij},F_2^{ij},F_3^{ij})\equiv 0$ for $\forall |\alpha|=1$ or  $\Phi^{\alpha}(F_1^{ji},F_2^{ji},F_3^{ji})\equiv 0$ for $\forall |\alpha|=1$. Hence we see that $j\in S(i)$ if and only if $i\in S(j)$. By Lemma \ref{lem2.4}, it is suffice for us to show that there exists an index $i$ such that $\sharp S(i)\ge \left [\frac{q}{2}\right]$.

Indeed, suppose contrarily that $\sharp S(i)< \left [\frac{q}{2}\right]$. Consider the simple graph $\mathcal G$ with vertices $\{1,\ldots,q\}$ and the set of edges consisting of all pair $\{i,j\}$ so that $j\not\in S(i)$. Therefore each vertex of this graph has degree at least $(q-1)-(\left [\frac{q}{2}\right]-1)\ge\dfrac{q}{2}$. Then by Dirac's theorem for simple graph, there exists a Hamilton cycle $i_1i_2\ldots i_qi_1$ in $\mathcal G$, for instance we suppose that $i_j=j$. Setting 
$$ \sigma (i)=\begin{cases}
i+1&\text{ if }i<q\\
1&\text{ if }i=q,
\end{cases} $$
we have $\sigma (i)\not\in S(i)$ and $i\not\in S(\sigma (i))$. Then by Lemma \ref{lem3.2}, we get functions $g_{\{i,\sigma (i)\}}$ corresponding to the pair $\{i,\sigma (i)\}$ and hence
$$\nu_{g_{\{i,\sigma (i)\}}}\ge 2\sum_{u=1}^3\sum_{t=i,\sigma (i)}\nu^{[n]}_{(f^u,H_t)}+4\sum_{\overset{t=1}{t\ne i,\sigma(i)}}^{q}\nu^{[1]}_{(f,H_t)}-(3n+2)\sum_{t=i,\sigma (i)}\nu^{[1]}_{(f,H_t)}$$
Summing both sides of the above inequalities over all $i=1,\ldots,q$, we get
\begin{align*}
\nu_{\prod_{i=1}^qg_{\{i,\sigma (i)\}}}&\ge 4\sum_{u=1}^3\sum_{i=1}^q\nu^{[n]}_{(f^u,H_t)}+(4q-6n-12)\sum_{i=1}^q\nu^{[1]}_{(f,H_t)}\\
&\ge\left (4+\frac{4q-6n-12}{3n}\right)\sum_{u=1}^3\sum_{i=1}^q\nu^{[n]}_{(f^u,H_t)}=\frac{4q+6n-12}{3n}\sum_{u=1}^3\sum_{i=1}^q\nu^{[n]}_{(f^u,H_t)}.
\end{align*}
It is clear that $\prod_{i=1}^qg_{\{i,\sigma (i)\}}\in B(2q,2q;f^1,f^2,f^3)$. Then, from Proposition \ref{prop2.9}, we have
\begin{align*}
q&\le n+1+\rho\frac{n(n+1)}{2}+\frac{3n}{4q+6n-12}(2q+\rho 2q)\\
&= n+1+\frac{3nq}{2q+3n-6}+\rho\left (\frac{n(n+1)}{2}+\frac{3nq}{2q+3n-6}\right).
\end{align*}
This is a contradiction.

Therefore, there exists $i_0$ such that $\sharp S(i_0)\ge \left [\frac{q}{2}\right]$. 
The theorem is proved. 
\end{proof}
\begin{claim}\label{cl3.9} If $n,q$ satisfy
\begin{align}\label{3.10}
q\ge\frac{n+6+(7n^2+2n+4)^{1/2}}{2}+\left (\rho\frac{3n((n+1)(q+n-3)+q-2)}{2}\right)^{1/2}
\end{align}
then
\begin{align}\label{3.11}
q> n+1+\frac{3nq}{2q+3n-6}+\rho\left (\frac{n(n+1)}{2}+\frac{3nq}{2q+3n-6}\right).
\end{align}
\end{claim}
\begin{proof}
We see that (\ref{3.11}) equivalent to the following:
\begin{align*}
&(q-n-1)(q+\dfrac{3n}{2}-3)-\dfrac{3nq}{2}>\dfrac{3n\left ((n+1)(\dfrac{q}{3}+\dfrac{n}{2}-1)+q\right)\rho}{2}\\ 
\Leftrightarrow&\ q^2-nq-4q-\dfrac{3n^2}{2}+\dfrac{3n}{2}+3>\dfrac{3n\left ((n+1)(\dfrac{q}{3}+\dfrac{n}{2}-1)+q\right)\rho}{2}\\
\Leftrightarrow&\left (q-\dfrac{n+4}{2}\right)^2> \dfrac{7n^2+2n+4}{4}+\dfrac{3n\left ((n+1)(\dfrac{q}{3}+\dfrac{n}{2}-1)+q\right)\rho}{2}.
\end{align*}
We note that $(n+1)\left (\dfrac{q}{3}+\dfrac{n}{2}-1\right )+q\le (n+1)(q+n-3)+q-2$, then from (\ref{3.10}) we imply that
\begin{align*}
\left(q-\dfrac{n+4}{2}\right)^2&>\left (q-\dfrac{n+6}{2}\right)^2>\dfrac{7n^2+2n+4}{4}+\rho\dfrac{3n((n+1)(q+n-3)+q-2)\rho}{2}\\ 
&\ge \dfrac{7n^2+2n+4}{4}+\rho\dfrac{3n((n+1)(\dfrac{q}{3}+\dfrac{n}{2}-1)+q)\rho}{2}, 
\end{align*}
and hence get (\ref{3.11}).
The claim is proved.
\end{proof}

\begin{proof}[{\sc Proof of theorem \ref{thm1.1}}]  Suppose contrarily that the theorem does not hold. By Lemma \ref{lem3.8} and Claim \ref{cl3.9}, there exist $([\frac{q}{2}]+1)$ hyperplanes among $H_i'$s, for instance they are $H_1,\ldots,H_p$ $(p=[\frac{q}{2}]+1)$, and for each $j\ (2\le j\le p)$ there exist constants $\alpha_j,\beta_j,$ not all zeros, such that
\begin{align*}
&\alpha_j\left (\frac{(f_1,H_j)}{(f_1,H_1)}-\frac{(f_2,H_j)}{(f_2,H_1)}\right )=\beta_j \left (\frac{(f_1,H_j)}{(f_1,H_1)}-\frac{(f_3,H_j)}{(f_3,H_1)}\right )\\
\mathrm{or}\ \ &\alpha_j\left (\frac{(f_1,H_1)}{(f_1,H_j)}-\frac{(f_2,H_1)}{(f_2,H_j)}\right )=\beta_j \left (\frac{(f_1,H_1)}{(f_1,H_j)}-\frac{(f_3,H_1)}{(f_3,H_j)}\right ).
\end{align*}
By the supposition, there exists an index $j\ (2\le j\le p)$, for instance $j=2$, such that $\alpha_2\ne 0$, $\beta_2\ne 0$ and $\alpha_2\ne \beta_2$. Thus
$$ \alpha_2\left (\frac{(f_1,H_2)}{(f_1,H_1)}-\frac{(f_2,H_2)}{(f_2,H_1)}\right )=\beta_2 \left (\frac{(f_1,H_2)}{(f_1,H_1)}-\frac{(f_3,H_2)}{(f_3,H_1)}\right ),$$
i.e.,
\begin{align}\label{3.12}
(\beta_2-\alpha_2)\frac{(f_1,H_2)}{(f_1,H_1)}+\alpha_2\frac{(f_2,H_2)}{(f_2,H_1)}=\beta_2\frac{(f_3,H_2)}{(f_3,H_1)}.
\end{align}
We denote by $S$ the set of all singularities of $f^{-1}(H_t)\ (1\le t\le q)$ and set 
$$I=S\cup \bigcup_{s\ne t}\left (f^{-1}(H_s)\cap f^{-1}(H_t)\right ).$$
Then $I$ is an analytic subset of codimension at least two in $\B^m(1)$. 

From (\ref{3.12}), it is easy to see that
\begin{align}\label{3.13}
\begin{split}
\nu_{(f_k,H_2)}(z)\ge\min\{\nu_{(f_l,H_2)}(z),\nu_{(f_s,H_2)}(z)\},\\
\nu_{(f_k,H_1)}(z)\le\max\{\nu_{(f_l,H_1)}(z),\nu_{(f_s,H_1)}(z)\},
\end{split}
\end{align}
for all $z\notin I$ and permutations $(k,l,s)$ of $\{1,2,3\}$.

We consider the meromorphic mapping $F$ of $\B^m(1)$ into $\P^1(\C)$ with a reduced representation $F=\left (\dfrac{h_1}{h_2}\dfrac{(f_1,H_2)}{(f_1,H_1)}:\dfrac{h_1}{h_2}\dfrac{(f_2,H_2)}{(f_2,H_1)}\right )$, where $h_1,h_2$ are holomorphic functions which are chosen so that
\begin{align*}
\nu_{h_1}(z)&=\max_{1\le u\le 3}\nu_{(f_u,H_1)}(z),\\ 
\nu_{h_2}(z)&=\min_{1\le u\le 3}\nu_{(f_u,H_2)}(z) 
\end{align*}
for all $z\not\in I$. This implies that  $g:=\dfrac{h_2(f_1,H_1)(f_2,H_1)}{h_1}$ is a holomorphic function. Setting $F_0=\dfrac{h_1}{h_2}\dfrac{(f_1,H_2)}{(f_1,H_1)}$ and $F_1=\dfrac{h_1}{h_2}\dfrac{(f_2,H_2)}{(f_2,H_1)}$, we have
\begin{align*}
||F||=\left (|F_0|^2+|F_1|^2\right)^{1/2}&=\frac{1}{|g|}\left (|(f_1,H_2)(f_2,H_1)|^2+|(f_2,H_2)(f_1,H_1)|^2\right)^{1/2}\\
&\le C\frac{||f_1||\cdot ||f_2||}{|g|},
\end{align*}
where $C$ is a positive constant. This implies that $T(r,F)\le\sum_{u=1}^3T(r,f^u)$. Therefore $S(l_0;F)\subset S(l_0;f^1,f^2,f^3)$ and $B(k,l_0;F)\subset B(k,l_0;f^1,f^2,f^3)$ for every $l_0$ and $k$.

Setting $F_2=\dfrac{h_1}{h_2}\dfrac{(f_3,H_2)}{(f_3,H_1)}$, we have $(\beta_2-\alpha_2)F_0+\alpha_2F_1=\beta_2F_2.$ Since $(F_0:F_1)$ is a reduced representation, $F_0,F_1,F_2$ has no common zero outside the indeterminacy set $I(F)$ of $F$, which is of codimension two. Also each zero of $F_i\ (0\le i\le 2)$ must be zero of $(f^u,H_1)$ or zero of $(f^u,H_2)\ (1\le u\le 3)$. Then we have
$$ \sum_{i=0}^2\nu^{[1]}_{F_i}\le\nu^{[1]}_{(f,H_1)}+\nu^{[1]}_{(f,H_2)}. $$

We suppose that $F$ is not constant, then there is an index set $\alpha\in (\Z_+)^m$ with $|\alpha|=1$ such that 
$$ W^{\alpha}(F)=\left |
\begin{array}{ccc}
F_0&F_1\\ 
\mathcal D^{\alpha}F_0&\mathcal D^{\alpha}F_1\\
\end{array}
\right|\not\equiv 0. $$
Then $\left |\dfrac{W^{\alpha}(F)}{F_sF_t}\right |\in S(1;F)$ for every $0\le s<t\le 2$. 

On the other hand, setting $P=\dfrac{(F_0-F_1)W^{\alpha}(F)}{F_0F_1F_2}$, we have
$$ |P| \le C\left (\left |\frac{W^{\alpha}(F)}{F_0F_2}\right |+\left |\frac{W^{\alpha}(F)}{F_1F_2}\right |\right),$$
for some positive constant $C$. Hence $P\in S(1;F)\subset S(1;f^1,f^2,f^3)$. As the usual property of wronskian, we have 
$$ \nu_P\ge\nu_{F_0-F_1}-\sum_{i=0}^2\nu^{[1]}_{F_i}\ge\sum_{i=3}^q\nu^{[1]}_{(f,H_i)}-\nu^{[1]}_{(f,H_1)}-\nu^{[1]}_{(f,H_2)}.$$

We consider the set of indices $\{3,4,\ldots,q\}$. For each $i\in\{3,4,\ldots,q\}$, denote by $S'(i)$ the set consisting of all $j\in\{3,4,\ldots,q\}\setminus\{i\}$ satisfying $\dfrac{(f^1,H_i)}{(f^1,H_j)}\equiv\dfrac{(f^2,H_i)}{(f^2,H_j)}$. if there exists an index $i$ such that $\sharp S'(i)\ge [q/2]-1$ then the conclusion of the theorem holds trivially. Therefore, we only consider the case where $\sharp S'(i)\le [q/2]-2$ for all $i$. Denote by $\mathcal H$ the simple graph with the set of vertices $\{3,\ldots,q\}$ and the set of edges consisting of all pair $\{i,j\}$ so that $j\not\in S'(i)$ (equivalent to $i\not\in S'(j)$). Then each vertex of $\mathcal H$ has degree at least $\left (q-3-\left[\dfrac{q}{2}\right]+2\right)\ge \dfrac{q-2}{2}$. Hence, by Dirac's theorem, the graph $\mathcal H$ contains a Hamilton cycle $j_1\ldots j_{q-2}j_1$. We set $v(i)=j_{i+1}$ if $i<q-2$ and $v(q-2)=j_1$, and
$$ P_i:=(f^1,H_{j_i})(f^2,H_{j_{v(i)}})- (f^1,H_{j_{v(i)}})(f^2,H_{j_i})\not\equiv 0\ (1\le i\le q-2).$$
Then we see that
\begin{align*}
\nu_{P_i}&\ge \min\{\nu_{(f^1,H_{j_i})},\nu_{(f^1,H_{j_i})}\}+\min\{\nu_{(f^1,H_{j_{v(i)}})},\nu_{(f^1,H_{j_{v(i)}})}\}+\sum_{\underset{s\ne j_i,j_{v(i)}}{s=1}}^q\nu^{[1]}_{(f,H_s)}\\
&\ge \sum_{s=j_i,j_{v(i)}}(\nu_{(f^1,H_s)}^{[n]}+\nu_{(f^2,H_s)}^{[n]}-(n+1)\nu_{(f,H_s)}^{[1]})+\sum_{s=1}^q\nu^{[1]}_{(f,H_s)}.
\end{align*}
Here, we use the inequality $\min\{a,b\}\ge\min\{a,n\}+\min\{b,n\}-n$. 

Setting $P_{\{1,2\}}=\prod_{i=1}^{q-2}P_i$ and summing up both sides of the above inequalities over all $i=1,\ldots,q-2$, we get
\begin{align*}
\nu_{P_{\{1,2\}}}&= 2\sum_{u=1,2}\sum_{i=3}^q\nu_{(f^u,H_s)}^{[n]}-2(n+1)\sum_{i=3}^q\nu_{(f^u,H_1)}^{[1]}+(q-2)\sum_{i=1}^q\nu^{[1]}_{(f,H_i)}\\
&\ge 2\sum_{u=1,2}\sum_{i=3}^q\nu_{(f^u,H_s)}^{[n]}+(q-2n-4)\sum_{i=3}^q\nu_{(f^u,H_1)}^{[1]}+(q-2)\left(\nu^{[1]}_{(f,H_1)}+\nu^{[1]}_{(f,H_2)}\right).
\end{align*}
Similarly, we define $P_{\{1,3\}},P_{\{2,3\}}$ and set $Q=P_{\{1,2\}}P_{\{1,3\}}P_{\{2,3\}}$. Then we have 
$$Q\in B(2(q-2),0;f^1,f^2,f^3))$$
and
$$ \nu_Q\ge 4\sum_{u=1}^3\sum_{i=3}^q\nu_{(f^u,H_s)}^{[n]}+3(q-2n-4)\sum_{i=3}^q\nu_{(f^u,H_1)}^{[1]}+3(q-2)(\nu^{[1]}_{(f,H_1)}+\nu^{[1]}_{(f,H_2)}).$$
Hence
\begin{align*}
\nu_{QP^{3(q-2)}}&\ge 4\sum_{u=1}^3\sum_{i=3}^q\nu_{(f^u,H_s)}^{[n]}+3(2q-2n-6)\sum_{i=3}^q\nu_{(f^u,H_1)}^{[1]}\\
&\ge \left (4+\frac{2q-2n-6}{n}\right)\sum_{u=1}^3\sum_{i=3}^q\nu_{(f^u,H_s)}^{[n]}.
\end{align*}
Since $QP^{3(q-2)}\in B(2(q-2),3(q-2);f^1,f^2,f^3)$, from the above inequality and Proposition \ref{prop2.9} we have
\begin{align*}
q-2&\le n+1+\rho\dfrac{3n(n+1)}{2}+\dfrac{1}{4+\frac{2q-2n-6}{n}}\left (2(q-2)+\rho 3(q-2)\right)\\
&=n+1+\dfrac{(q-2)n}{q+n-3}+\rho\left (\dfrac{3n(n+1)}{2}+\dfrac{3(q-2)n}{2q+2n-6}\right).
\end{align*}
Thus
$$ (q-n-3)(q+n-3)-(q-2)n\le \rho\dfrac{3n(n+1)(q+n-3)+3(q-2)n}{2}, $$
i.e., 
$$ \left (q-3-\frac{n}{2}\right)^2\le \frac{5n^2+4n}{4}+\dfrac{3n((n+1)(q+n-3)+q-2)\rho}{2}.$$
This implies that
$$q-3-\frac{n}{2}\le \frac{(5n^2+4n)^{1/2}}{2}+\left (\rho\dfrac{3n((n+1)(q+n-3)+q-2)}{2}\right )^{1/2}.$$
Thus
\begin{align*}
q\le&\frac{n+6+(5n^2+4n)^{1/2}}{2}+\left (\rho\dfrac{3n((n+1)(q+n-3)+q-2)}{2}\right )^{1/2}\\ 
<&\frac{n+6+(7n^2+2n+4)^{1/2}}{2}+\left (\rho\dfrac{3n((n+1)(q+n-3)+q-2)}{2}\right )^{1/2}.
\end{align*}
This is a contradiction. Then the supposition is false. The theorem is proved.
\end{proof}

\begin{proof}[{\sc Proof of Theorem \ref{thm1.2}}]
Without loss of generality, in this proof we only consider the case where $M=\C^m$ and the case where $M=\mathbb B^m(1)$.

Denote by $\mathcal I$ the set of all $k$-tuple $I=(i_1,\ldots,i_k)\in \N^k$ with $1\le i_1<i_2<\cdots <i_k\le q$ and set $p=\sharp\mathcal I$.

Suppose contrarily that $f^1\times f^2\times\cdots\times f^k$ is not algebraically degenerate. Then for every $I=(i_1,\ldots,i_k)\in\mathcal I$,
$$ P_{I}:=\det \left ((f^s,H_{i_t}); 1\le s,t\le k\right)\not\equiv 0. $$
By Lemma \ref{lem2.4}, we have
\begin{align*}
\nu_{P_I}&\ge\sum_{s=1}^k\left (\min\{\nu_{(f^u,H_{i_s})};1\le u\le k\}-\nu^{[1]}_{(f,H_{i_s})}\right)+(k-1)\sum_{i=1}^q\nu^{[1]}_{(f,H_i)}\\
&=\sum_{s=1}^k\left (\nu^{[n]}_{(f,H_{i_s})}-\nu^{[1]}_{(f,H_{i_s})}\right )+(k-1)\sum_{i=1}^q\nu^{[1]}_{(f,H_i)}.
\end{align*}
Setting $P=\prod_{I\in\mathcal I}P_I$ and summing up both sides of the above inequalities over all $I\in\mathcal I$, we get
\begin{align}\label{3.14}
\begin{split}
\nu_P&\ge\sum_{i=1}^q\left (\frac{pk}{q}\nu^{[n]}_{(f,H_i)}+\frac{p((k-1)q-k)}{q}\nu^{[1]}_{(f,H_i)}\right)\\
&\ge\left (\frac{pk}{q}+\frac{p((k-1)q-k)}{nq}\right)\sum_{i=1}^q\nu^{[n]}_{(f,H_i)}\\
&=\left (\frac{p}{q}+\frac{p((k-1)q-k)}{knq}\right)\sum_{u=1}^k\sum_{i=1}^q\nu^{[n]}_{(f^u,H_i)}.
\end{split}
\end{align}

Case 1. $M=\C^m$. By the second main theorem (see \cite{NO}), we have
\begin{align*}
\biggl \|\ (q-n-1)\sum_{u=1}^kT_{f^u}(r,1)\le&\sum_{u=1}^k\sum_{i=1}^qN^{[n]}_{(f^u,H_i)}(r,1)+o\left (\sum_{u=1}^kT_{f^u}(r,1)\right )\\ 
\le& \left (\frac{p}{q}+\frac{p((k-1)q-k)}{knq}\right)^{-1}N_P(r,1)+o\left (\sum_{u=1}^kT_{f^u}(r,1)\right )\\
\le&\frac{knq}{kn+(k-1)q-k}\sum_{u=1}^kT_{f^u}(r,1)++o\left (\sum_{u=1}^kT_{f^u}(r,1)\right ).
\end{align*}
Here, the notation ``$\|$'' means the inequality holds for all $r\in [1,+\infty)$ outside a finite Lebesgue measure set. 
Letting $r\rightarrow +\infty$, we get
$$ q\le n+1+\frac{knq}{kn+(k-1)q-k}.$$
This is a contradiction.

Case 2. $M=\mathbb B^m (1)$. Applying Proposition \ref{prop2.9} for the function $P\in B(p,0;f^1,\ldots,f^k)$, we get
$$ q\le n+1+p\left (\frac{p}{q}+\frac{p((k-1)q-k)}{knq}\right)^{-1}+\rho k\frac{n(n+1)}2, $$
i.e.,
$$ q\le n+1+\frac{knq}{kn+(k-1)q-k}+\rho\frac{kn(n+1)}{2}.$$
This is a contradiction.

Hence, the supposition is false. The theorem is proved.
\end{proof}


\end{document}